\documentclass [a4paper, 12pt]{amsart}

\usepackage[margin=2.5cm]{geometry} 
\usepackage[alphabetic]{amsrefs}
\usepackage{wasysym,stmaryrd,enumerate}

\usepackage{amsthm}
\usepackage{latexsym}
\usepackage{times}
\usepackage{graphicx}
\usepackage{yhmath}

\usepackage{amsmath}
\usepackage{amsfonts}
\usepackage{amssymb}
\usepackage{amsthm}
\usepackage{graphicx}
\usepackage{pb-diagram}
\usepackage{epstopdf}
\usepackage{fancyhdr}
\usepackage[all]{xy}
\usepackage{cleveref}

\newtheorem{cor}{Corollary}[section]
\newtheorem{op}[cor]{Open Problem}
\newtheorem{te}[cor]{Theorem}
\newtheorem{p}[cor]{Proposition}

\newtheorem{lemma}[cor]{Lemma}
\theoremstyle{definition}
\newtheorem{de}[cor]{Definition}
\theoremstyle{remark}
\newtheorem{ob}[cor]{Observation}
\newtheorem{ex}[cor]{Example}

\newtheorem{nt}[cor]{Notation}

\newcommand{\cz}{\mathbb{C}}
\newcommand{\nz}{\mathbb{N}}

\newcommand{\rz}{\mathbb{R}}

\newcommand{\ff}{\mathbb{F}}

\newcommand{\unit}{\mathcal{U}}

\newcommand{\vp}{\varphi}
\newcommand{\ve}{\varepsilon}

\newcommand{\sm}{\setminus}

   {\begin{verse}%
     \small }%
   {\end{verse}}

\def\tilde{\widetilde}

\def\cc{{\mathcal C}}

\begin{document}

\title{On Krein-Milman theorem for the space of sofic representations}

\author{Radu B. Munteanu}
\address[R.B. Munteanu]{Department of Mathematics, University of Bucharest, 14 Academiei Street and Institute of Mathematics of the Romanian Academy, 21 Calea Grivitei Street, 010702 Bucharest, Romania}
\email{radu-bogdan.munteanu@g.unibuc.ro}

\author{Liviu P\u aunescu}
\address[L. P\u aunescu]{Institute of Mathematics of the Romanian Academy, 21 Calea Grivitei Street, 010702 Bucharest, Romania}
\email{liviu.paunescu@imar.ro}
\maketitle

%\subjclass[2010]{{20Fxx, 20F05, 20B30, 15A27}} \keywords{Residually finite groups, sofic groups, almost commuting matrices, ultraproducts.}

\begin{abstract}
Denote by $Sof(G)$ the space of sofic representations of a countable group $G$. This space is known by a result of the second author, to have a convex-like structure. We show that, in this space, minimal faces are extreme points. We then construct uncountable many extreme points for $Sof(\ff_2)$ and show that there exists a decreasing chain of closed faces with empty intersection. Finally we construct a strangely looking sofic representation in $Sof(\mathbb{F}_2)$ that we believe it is outside of the closure of the convex hull of extreme points.
\end{abstract}

\footnotetext[1]{This work was supported by a grant of the Romanian Ministry of Education and Research, CNCS - UEFISCDI, project number PN-III-P1-1.1-TE-2019-0262, within PNCDI III. The first author was partially supported by PN-III-P4-ID-PCE-2020-2693 grant from CNCS - UEFISCDI Romania.}

The starting point of our discussion is the paper \cite{Br}. Nate Brown considered the space of all morphisms from a fixed finite von Neumann algebra $N$ to the ultrapower of the hyperfinite factor, up to unitary equivalence, denoted by $Hom(N,R^\omega)$. He then constructed a convex structure on this space, and studied its extreme points. For a sofic group $G$, the space $Sof(G,P^\omega)$ is constructed analogously. Similar properties hold for extreme points, as shown in \cite{Pa2,Pa3}.

Over the years these spaces have been studied by different authors, possibly using other frameworks. In \cite{Ju}, it is shown that $Hom(G,R^\omega)$ consists of one point if and only if $G$ is amenable. Elek and Szabo showed the same thing for $Sof(G,P^\omega)$ in \cite{El-Sz}. Capraro and Fritz, \cite{Ca-Fr}, showed that these spaces can be embedded in an abstractly constructed Banach space, while Atkinson, \cite{Ak}, studied finite dimensional faces of $Hom(G,R^\omega)$. In \cite{Pa3}, the second author showed that there are groups $H\subset G$ such that the restriction $R:Sof(G,P^\omega)\to Sof(H,P^\omega)$ is not surjective. This is an obstruction to soficity, that hints at the existence of non-sofic groups. Let's state the main question to be investigated in this work.

\vspace{0.2cm}
\textbf{Question} \emph{Do $Hom(N,R^\omega)$ and $Sof(G,P^\omega)$ satisfy a Krein-Milman result?}
\vspace{0.2cm}

Apart from the original work of Brown and Capraro \cite{Br,Br-Ca}, this problem has been tackled before by Chirvasitu in \cite{Ch}. The author hints at an argument in favour of a positive answer (see discussion before Propostion 2.10 of \cite{Ch}). Let's recall the Krein-Milman theorem.

\vspace{0.2cm}
\textbf{Theorem} (Krein-Milman) \emph{Let $X$ be a locally convex topological vector space, and let $K$ be a compact convex subset of $X$. Then $K$ is the closed convex hull of its extreme points.}
\vspace{0.2cm}

$Hom(N,R^\omega)$ and $Sof(G,P^\omega)$ are subsets in a Banach space (that is always locally convex). However, in the appendix of \cite{Br}, Ozawa showed that they are never compact (unless $N$, or $G$ are amenable per Jung and Elek-Szabo's results). For general non-compact convex sets, the Krein-Milman theorem fails. 

In the proof of the theorem, compactness is used in two places: to show that minimal faces are points, and to deduce that a decreasing chain of closed faces has non-trivial intersection. In Section \ref{minimal faces}, we show that minimal faces of $Sof(G,P^\omega)$ are points, as before, even thought $Sof(G,P^\omega)$ is not compact. We then construct a decreasing chain of closed faces, with trivial intersection, Section \ref{decreasing faces}. Finally we construct a sofic representation of the free group, with trivial commutant and uncountably many cuts, Section \ref{another}

\section{Introduction}

For a matrix $x\in M_n$ we define its normalised trace as $Tr(x)=\frac1n\sum_ix(i,i)$. Throughout the paper we denote by $P_n\subset M_n$ the group of permutation matrices, and $D_n\subset M_n$ the maximal abelian subalgebra of diagonal matrices. The group $P_n$ is isomorphic to $Sym(n)$, the symmetric group on a set of $n$ elements. For simplicity, we call elements in $P_n$ permutations, instead of permutation matrices. For $p\in P_n$, $Tr(p)=1-\frac1n|Fix(p)|=1-d_H(p,Id)$, where $d_H$ is the normalised Hamming distance. On the algebra $D_n$, $Tr:D_n\to\cz$ is acting as an integral and $D_n$ is isomorphic to $L^\infty(\{1,\ldots,n\})$ endowed with the normalised cardinal measure.

Fix now $\omega$ a free ultrafilter on $\nz$ and $(n_k)_k$ a sequence of natural numbers, $\lim_kn_k=\infty$. The ultraproduct $\Pi_{k\to\omega}P_{n_k}$ is called \emph{the universal sofic group}. It was introduced by Elek and Szabo in \cite{El-Sz1}. They showed that a countable group is sofic if and only if it is a subgroup of $\Pi_{k\to\omega}P_{n_k}$.

The ultraproduct $(\Pi_{k\to\omega}D_{n_k},Tr)$ yields an abelian finite von Neumann algebra. As such, there exists a probability space $(X_\omega,\mu_\omega)$ such that $(\Pi_{k\to\omega}D_{n_k},Tr)\simeq L^\infty(X_\omega,\mu_\omega)$. We can construct $(X_\omega,\mu_\omega)$ as an ultraproduct of finite probability spaces, i.e. a Loeb construction. It comes with a measurable map called \emph{the standard part} $St:X_\omega\to [0,1]$. This map induces a canonical embedding $St^*:L^\infty([0,1],\mu)\to L^\infty(X_\omega,\mu_\omega)$, where $\mu$ is the Lebesgue measure. For more details on the Loeb space, standard part, and also the action of $\Pi_{k\to\omega}P_{n_k}$
on $(\Pi_{k\to\omega}D_{n_k},Tr)$, check Section 1 of \cite{CMP}.

We assume familiarity with the space of sofic representations for a group $G$, denoted by $Sof(G,P^\omega)$, and its convex-like structure. We refer the reader to Sections 2.1 and 2.2 of \cite{Pa2}. We use the same notations. 

\begin{p}[Proposition 2.2 of \cite{Pa2}]
Let $\Theta_i$, $i=1,\ldots,n$ be sofic representations of a group $G$, and $\lambda_i\in[0,1]$ be such that $\sum_{i=1}^n\lambda_i=1$. Then there exists a well defined element of $Sof(G,P^\omega)$ denoted by $\sum_{i=1}^n\lambda_i[\Theta_i]$ such that the axioms of convex-like structures are observed.  
\end{p}

The reverse operation to taking convex combinations is cutting with a commuting projection. Here we recall the construction of a cut sofic representation, as it is central to the paper.

\begin{de}
Let $\Theta:G\to\Pi_{k\to\omega}P_{n_k}$ be a sofic representation and $p\in\Pi_{k\to\omega}D_{n_kr_k}$ be a projection commuting with $\Theta\otimes Id_{r_k}$. We denote by $\Theta_p$ the map $p(\Theta\otimes Id_{r_k}):G\to\Pi_{k\to\omega}P_{m_k}$, where $m_k$ are natural numbers such that
$Tr(p)=\lim_{k\to\omega}\frac{m_k}{n_kr_k}$. The sofic representation $\Theta_p$ depends on the choice of these numbers, but its class in $Sof(G,P^\omega)$ does not.

We call such a $p\in\Pi_{k\to\omega}D_{n_kr_k}$ a \emph{cutting projection of $\Theta$}, and we denote the set of such projections by $\cc(\Theta)$.
\end{de}

\begin{ob}
$[\Theta_p]=[\Theta_{p\otimes Id_{r_k}}]$; $Tr(p)=Tr(p\otimes Id_{r_k})$.
\end{ob}

\begin{ob}(Observation 2.5 of \cite{Pa2})\label{obs_cut}
If $[\Theta]=\lambda[\Psi_1]+(1-\lambda)[\Psi_2]$ then there exists $p$ a cutting projection such that $[\Theta_p]=[\Psi_1]$.
\end{ob}

\begin{ob}\label{obs_doublecut}
If $p$ and $q$ are two cutting projections in the same sequence of dimensions, then $[(\Theta_p)_q]=[\Theta_{pq}]$.
\end{ob}

This last observation is just an algebraic consequence of the definition. However there is a discussion here to be made. When we write $a=\Pi_{k\to\omega}a_k\in\Pi_{k\to\omega}P_{n_k}$ we think of the permutations $a_k\in P_{n_k}$ as acting on the set $\{1,\ldots,n_k\}$. But, for $\Theta_p(g)=p(\Theta(g)\otimes Id_{r_k})$, $(\Theta_p(g))_k$ is a permutation on the support of $p_k$. As a projection in $D_{n_kr_k}$, $p_k$ is a projection on some subset of $\{1,\ldots,n_kr_k\}$. It is on these subsets that $\Theta_p(g)$ is acting as an ultraproduct of permutations. Having this in mind, now $(\Theta_p)_q$ makes sense. 

\subsection{Infinite convex combinations}

In this paper we also use infinite convex combinations. They are constructed similarly as in Section 2.2 of \cite{Pa2}, plus a diagonal argument.

\begin{p}\label{infinite}
Let $(\Theta_i)_{i\in\nz^*}$ be a sequence of sofic representations of a countable group $G$, and let $\lambda_i\in[0,1]$ be such that $\sum_{i\in\nz^*}\lambda_i=1$.
Then there exist a sofic representation $\Psi:G\to\Pi_{k\to\omega}P_{m_k}$ such that for each $i\in\nz^*$ there exists a cutting projection $p_i\in\Pi_{k\to\omega}D_{m_k}$ such that $Tr(p_i)=\lambda_i$, $\sum_ip_i=Id$ and $[\Psi_{p_i}]=[\Theta_i]$.
\end{p}

\subsection{Order relation on $\cc(\Theta)$} There is a partial relation for cutting projections. Firstly we consider a projection equivalent to any of its amplifications, i.e. $p\simeq p\otimes Id_{r_k}$.

\begin{de}
Let $\Theta$ be some sofic representation and let $p,q\in\cc(\Theta)$. Then $p\leqslant q$ if they have amplifications $p_1,q_1$ to the same sequences of dimensions such that $p_1q_1=p_1$.
\end{de}

Let $\Theta:G\to\Pi_{k\to\omega}P_{n_k}$. Sometimes is useful to view a cutting projection of $\Theta$ as an element of $\Pi_{k\to\omega}D_{n_k}$ instead of using amplifications. For a finite von Neumann algebra $(N,Tr)$ we denote by $N_+^1=\{x\in N:0\leqslant x\leqslant 1\}$. We also need a notation for the \emph{commutant}, a central tool of the article.

\begin{nt}
For a map $\Theta:G\to\Pi_{k\to\omega}P_{n_k}$, we denote by $\Theta^\prime$ the commutant of the image of $\Theta$ in $\Pi_{k\to\omega}M_{n_k}$, i.e.:
\[\Theta^\prime=\{x\in\Pi_{k\to\omega}M_{n_k}:xa=ax\ \forall a\in\Theta(G)\}.\]
\end{nt}

\begin{p}
For a sofic representation $\Theta:G\to\Pi_{k\to\omega}P_{n_k}$ we have $\cc(\Theta)\simeq(\Pi_{k\to\omega}D_{n_k})_+^1\cap\Theta^\prime$.
\end{p}
\begin{proof}
Let $p\in\Pi_{k\to\omega}D_{n_kr_k}$ be a projection commuting with $\Theta\otimes Id_{r_k}$. Then $p=\Pi_{k\to\omega}p_k$, where $p_k$ is a projection in $D_{n_kr_k}$. As such $p_k=\sum_{i=1}^{r_k} p_k^i$, with $p_k^i\in D_{n_k}$. Define $f_k=\frac1{r_k}\sum_ip_k^i$, and $f=\Pi_{k\to\omega}f_k$. It is easy to see that $0\leqslant f\leqslant Id$. Also $p$ commuting with $\Theta$ implies $f$ commutes with $\Theta$.

For the reverse, is an easy analysis exercise to construct a cutting projection given an element in $(\Pi_{k\to\omega}D_{n_k})_+^1\cap\Theta^\prime$.
\end{proof}

\begin{ob}
For $p,q\in\cc(\Theta)$, we have $p\leqslant q$ if and only if $f_p\leqslant f_q$, where $f_p,f_q$ are the associated elements in $\Pi_{k\to\omega}D_{n_k}$.
\end{ob}

\section{The face of a sofic representation}

This section is dedicated to the study of faces of $Sof(G,P^\omega)$ given by a sofic representation.

\begin{de}
We denote by $F_{[\Theta]}$ the set $\{[\Theta_p]:  p\text{ a cutting projection of }\Theta \}.$
\end{de}
The goal of this section is to prove that  $F_{[\Theta]}$ is the  minimal face containing $[\Theta]$. This is an adaptation of results in \cite{Ak} to the convex structure on sofic embeddings. First, some preliminaries.

\begin{lemma}\label{lemma_2proj}
If $p,q$ are disjoint cutting projections in the same sequence of dimensions then: 
\[[\Theta_{p+q}]= \frac{Tr(p)}{Tr(p+q)}[\Theta_{p}]+\frac{Tr(q)}{Tr(p+q)}[\Theta_{q}].\]
\end{lemma}
\begin{proof} 
Let $p,q\in\Pi_{k\to\omega}D_{n_kr_k}$. Then $\Theta_{p+q}=(\Theta\otimes Id_{r_k})_p\oplus(\Theta\otimes Id_{r_k})_q$. As $\Theta_{p+q}$ is a direct sum, by the definition of the convex structure, its class is a convex combination of its summands. Thus $[\Theta_{p+q}]$ is a convex combination of $[\Theta_p]$ and $[\Theta_q]$. The coefficients of this convex combination are given by how much space $[\Theta_p]$ and $[\Theta_q]$ occupy in the direct sum. A close inspection of the dimensions of the permutations involved, yields the stated result.
\end{proof}

The following result is important, as it shows that $F_{[\Theta]}$ is convex, a first requirement of being a face.

\begin{lemma}\cite[Analogue of Proposition 3.3]{Ak}
Let $\Theta: G\to\Pi_{k\to\omega}P_{n_k}$ be a sofic representation, $p,q$ be two cutting projections and $\lambda\in[0,1]$. Then there exists $s$ a cutting projection such that:
\[\lambda[\Theta_p]+(1-\lambda)[\Theta_q]=[\Theta_s].\]
\end{lemma}
\begin{proof}
We can assume that $p$ and $q$ are in the same sequence of dimensions, i.e. $p,q\in\Pi_{k\to\omega}D_{n_kr_k}$. Let $t\in\Pi_{k\to\omega}D_{n_k}$ be any projection such that:
\[Tr(t)= \frac{\lambda Tr(q)}{\lambda Tr(q)+(1-\lambda)Tr(p)}.\] 
This value is chosen such that $\lambda=\frac{Tr(p)Tr(t)}{Tr(p)Tr(t)+Tr(q)(1-Tr(t))}$. Construct $s\in\Pi_{k\to\omega}D_{n_k^2r_k}$ by:
\[s=p\otimes t+ q\otimes(1-t).\] 
It is easy to see that $s$ is a cutting projection. Moreover: 
\[\Theta_s=\Theta_{p\otimes t}\oplus\Theta_{q\otimes (1-t)}=(\Theta\otimes Id_{r_k})_p\otimes(Id_{n_k})_t\oplus (\Theta\otimes Id_{r_k})_q\otimes(Id_{n_k})_{1-t}.\]
As in the previous lemma, $[\Theta_s]$ is a convex combination between $[\Theta_p]$ and $[\Theta_q]$. The value of $Tr(t)$ is chosen such that this is the required combination, i.e. $[\Theta_s]=\lambda[\Theta_p]+(1-\lambda)[\Theta_q]$.
\end{proof}

\begin{p}\cite[Analogue of Proposition 3.3]{Ak}
The set $F_{[\Theta]}$ is the minimal face containing $[\Theta]$.
\end{p}
\begin{proof}
We first show that $F_{[\Theta]}$ is indeed a face. By the previous lemma, $F_{[\Theta]}$ is convex. Let now $\lambda[\Psi_1]+(1-\lambda)[\Psi_2]$ be an element of $F_{[\Theta]}$. So there is a cutting projection $p$ such that $[\Theta_p]= \lambda[\Psi_1]+(1-\lambda)[\Psi_2]$. By Observation \ref{obs_cut}, there exists $q$ a cutting projection such that $[(\Theta_p)_q]=[\Psi_1]$ and by Observation \ref{obs_doublecut}, $[(\Theta_p)_q]=[\Theta_{pq}]$ (amplifying $p$ and $q$ to the same sequence of dimensions if needed). Thus $[\Psi_1]=[\Theta_{pq}]\in F_{[\Theta]}$.

Let now $F$ be a face containing $[\Theta]$, and let $p$ be a cutting projection. As a particular case of Lemma \ref{lemma_2proj}, we get $[\Theta]=Tr(p)[\Theta_p]+ (1-Tr(p))[\Theta_{1-p}]$. It follows that $F$ contains $[\Theta_p]$, so it contains $F_{[\Theta]}$.
\end{proof}

\begin{te}
$F_{[\theta]}$ is closed.
\end{te}
\begin{proof}
This follows by a diagonal argument constructing a sofic representation from a sequence of sofic representations. It is also important that the space $Sof(G,P^\omega)$ is metric.
\end{proof}

\section{Minimal faces}\label{minimal faces}

The goal of this section is to prove that minimal faces are points, and thus extreme points. For a compact convex subset this is an easy consequence of Hahn-Banach theorem, and part of the proof of Krein-Milman theorem. Here we don't have compactness, so we deduce this result by other means.

\begin{p}\label{maxcut}
If $\Psi\in F_{[\Theta]}$ there exists a maximal cutting projection $s$ of $[\Theta]$ such that $[\Psi]= [\Theta_s]$.
\end{p}
\begin{proof}
First of all, when we say \emph{maximal} projection, we consider a projection to be equivalent to any of its amplifications. So if $p,q$ are two projections then $p\leqslant q$ if they have amplifications $p_1,q_1$ to the same sequences of dimensions such that $p_1q_1=p_1$.

Let 
$$A=\{p : p\mbox{ cutting projection, } [\Theta_p]=[\Psi]\}$$
endowed with the order relation specified above. Let  
 $P$ be a totally ordered subset of $A$. We want to prove that $P$ has an upper bound in $A$. The  important thing to notice here, is that for $p,q\in P$, $Tr(p)\leqslant Tr(q)$ implies $p\leqslant q$. This comes from the fact that $P$ is totally ordered. If $\sup\{ Tr(p):  p\in P\}$ is attained in $P$, then that element would be maximal in $P$ and we are done. Otherwise, choose $q_i\in P$, $i\in\mathbb{N}$  such that $q_i<q_{i+1}$ for $i\in\mathbb{N}$ and  
\begin{equation*}\label{eq1}
\sup\{Tr(q_i) :  i\in\mathbb{N} \}=\sup\{ Tr(p):  p\in P\}.
\end{equation*}
The goal here is to reduce $P$ to a countable subset. Indeed, for each $p\in P$ there is $i\in\nz$ such that $p<q_i$. All we have to do is to construct an upper bound in $A$ of the sequence $\{q_i\}_{i\in\nz}$.

Let $q_i=\Pi_{k\to\omega}q_i^k$. We know that $(q_i\otimes Id)(q_{i+1}\otimes Id)=q_i\otimes Id$. It is a rather involved diagonal argument here to show that there exists $p$ a cutting projection $q_i<p$ for any $i$ and $Tr(p)=\sup\{Tr(q_i) :  i\in\mathbb{N} \}$. Alternatively, we can simply take the supremum of the sequence $\{q_i\}_{i\in\nz}$, when these projections are embedded in a direct limit of amplifications, like the map $\Psi$ from Notation 2.2 in \cite{Pa3}.

We only need to show that $[\Theta_p]=[\Psi]$. As $q_i<p$, by Lemma \ref{lemma_2proj} and Observation \ref{obs_doublecut} we have:
\[[\Theta_p]=\frac{Tr(q_i)}{Tr(p)}[\Theta_{q_i}]+\frac{Tr(p)-Tr(q_i)}{Tr(p)}[\Theta_{p- q_i}].\]
As $\lim_{i\rightarrow\infty}Tr(q_i)=Tr(p)$, the axioms of the convex structure imply that 
\[\lim_{i\rightarrow\infty}d([\Theta_{p}], [\Theta_{q_i}])=0.\]
But $[\Theta_{q_i}]=[\Psi]$, as $q_i\in A$. We conclude that $[\Theta_{p}]=[\Psi]$. The proposition follows now by Zorn's lemma. 
\end{proof}

\begin{te}
The minimal faces are exactly the extremal points. 
\end{te}
\begin{proof}Let $F$ be a minimal face and let $\Theta\in F$. As $F$ is minimal, $F=F_{[\Theta]}$. Consider $[\Psi]\in F_{[\Theta]}$. By the previous Proposition, there exists a maximal projection  $q$ such that $[\Theta_q]=[\Psi]$. We will show that $q=1$.
Assume that $q<1$. We know that $[\Theta_{(1-q)}]\in F$. Thus $F_{[\Theta_{(1-q)}]}\subset F$ and since $F$ is minimal, we have $F=F_{[\Theta_{(1-q)}]}$. As $[\Psi]\in F_{[\Theta_{(1-q)}]}$ there exists a non zero projection $p$ such that $[\Psi]=[(\Theta_{(1-q)})_p]$. Then, by Observation \ref{obs_doublecut} $[\Psi]=[\Theta_{(1-q)p}]$ and by Lemma \ref{lemma_2proj} $[\Psi]=[\Theta_{q+(1-q)p}]$. This contradicts the maximality of $q$. So any element in $F$ is equal to $[\Theta]$.
\end{proof}

\begin{op}
Does this result follow directly from Nate Brown's axioms, for any convex-like sturcture?
\end{op}

\section{Decreasing chain of faces}\label{decreasing faces}

Recall that $[\Theta]\in Sof(G,P^\omega)$ is an extreme point if and only if $[\Theta]=[\Theta_p]$ for any cutting projection $p$ (Lemma 2.12 of \cite{Pa2}). We need a countable sequence of such extreme sofic representations in order to construct a decreasing chain of faces.

\begin{p}
Let $(\Theta_i)_{i\in\nz^*}$ be a sequence of different extreme sofic representations of a countable group $G$. Then there exist a sofic representation $\Psi$ such that:
\begin{enumerate}
\item For each $i\in\nz^*$ there exists $p_i\in\cc(\Psi)$ such that $[\Psi_{p_i}]=[\Theta_i]$;
\item For each $p\in\cc(\Psi)$ there exists $q\in\cc(\Psi)$, $q\leqslant p$ and $i\in\nz$ such that $[\Psi_q]=[\Theta_i]$.
\end{enumerate}
\end{p}
\begin{proof}
Let $(\lambda_i)_{i\in\nz^*}\subset[0,1]$ be a sequence such that $\sum_i\lambda_i=1$. Use Proposition \ref{infinite} to construct a sofic representation $\Psi:G\to\Pi_{k\to\omega}P_{n_k}$ with the given properties. Then point $(1)$ is automatic and does not require $\Theta_i$ to be extreme points.

Now let $p\in\cc(\psi)$, $p\in\Pi_{k\to\omega}D_{n_kr_k}$. Then there exists $i$ such that $Tr\big(p(p_i\otimes Id_{r_k})\big)\neq 0$. Let $q=p(p_i\otimes Id_{r_k})$. Then $q\leqslant p$ and $q\leqslant p_i$. As $[\Psi_{p_i}]=[\Theta_i]$ is an extreme point, it follows that $[\Psi_q]=[\Theta_i]$.
\end{proof}

\begin{te}
Let $(\Theta_i)_{i\in\nz^*}$ be a sequence of different extreme sofic representations of a countable group $G$. Then there exists a decreasing chain of faces in $Sof(G,P^\omega)$ with trivial intersection.
\end{te}
\begin{proof}
Let $[\Psi]$ be the sofic representation constructed above. By construction, we have cutting projections $(p_i)_{i\in\nz^*}$ such that $[\Psi_{p_i}]=[\Theta_i]$ and $\sum_ip_i=Id$. As $[\Theta_i]$ are different extreme points, $p_i\in\cc(\Psi)$ is maximal with the property $[\Psi_{p_i}]=[\Theta_i]$.

Let $q_j=\sum_{i\geqslant j}p_i$ and define the face $F_j=F_{[\Psi_{q_j}]}$. We show that $\cap_jF_j=\emptyset$. If $i<j$ then $[\Theta_i]\notin F_j$, otherwise $p_i$ would not be maximal. Assume that $\cap_jF_j$ is not empty. So there exists $p\in\cc(\Psi)$ such that $[\Psi_p]\in\cap_jF_j$. By the second property of the above Proposition, there exists $q\leqslant p$ and $i\in\nz^*$ such that $[\Psi_q]=[\Theta_i]$. But $q\leqslant p$ implies $[\Psi_q]\in\cap_jF_j$, so $[\Theta_i]\in\cap_jF_j$, contradiction.
\end{proof}

Of course this chain of decreasing faces does not help towards settling Krein-Milman for $Sof(G,P^\omega)$. We actually know for sure that $F_{[\Psi]}$ is inside the closure of the convex hull of extreme points. The following section constructs a sofic representation that is a candidate for an element outside of the convex hull.

For a sofic, non-amenable group $G$ there are two posibilities: either $Sof(G,P^\omega)$ has a finite number of extreme points, case in which Krei-Milman does not hold, as $Sof(G,P^\omega)$ is not separable; or $Sof(G,P^\omega)$ has an infinite number of extreme points, care in which, by the above theorem, it has a decreasing chain of faces with trivial intersection. In the next section we prove that the space $Sof(\ff_r,P^\omega)$ has uncountably many extreme points.

\section{Extreme points in $Sof(\ff_2,P^\omega)$}

In this section we construct uncountably many extreme points in the space $Sof(\ff_2,P^\omega)$. Each of these points will be an \emph{expander}, so let us first recall the definition. 

\begin{de}
A collection $p_1,\ldots,p_s\in P_n$ is called a \emph{$\lambda$-expander} if $\forall e\in D_n$ with $0<Tr(e)\leqslant\frac12$ we have $\lambda Tr(e)<\sum_{i=1}^sd_H(e,p_iep_i^*)$.
\end{de}

\begin{de}
A sofic representation $\Theta:G\to\Pi_{k\to\omega}P_{n_k}$  is called a \emph{expander} or \emph{$\lambda$-expander} if there exist $g_1,\ldots,g_s\in G$ and $p_k^i\in P_{n_k}$ such that $\Theta(g_i)=\Pi_{k\to\omega}p_k^i$ for $i=1,\ldots, s$ and $\{p_k^1,\ldots,p_k^s\}$ is an $\lambda$-expander for each $k\in\nz$.
\end{de}

From now on $a$ or $a_n$ denotes the matrix in $P_n$ associated to the cycle permutation $i\to i+1 (mod\ n)$. We also fix $\lambda=0.1$ for the rest of this section. The value is chosen such that $2\lambda<1/e$, as needed in the proof of Propostion \ref{k-permutations}. Also this value allows us to import some results from \cite{Pa3}. For example, taking two random elements in $P_n$ yields an $\lambda$-expander with great probability. For $n$-cycles this was already shown.

\begin{p}\cite[Proposition 5.11]{Pa3}
There exists a constant $\mu$ such that for at least $(1-\mu /n)\cdot (n-1)!$
$n$-cycles $c$ the following holds: for any projection $p\in D_n$ with $0 < Tr(p)\leqslant 1/2$ we have $\lambda Tr(p)< d_H(p, apa^*)+d_H(p, cpc^*)$.
\end{p}

This time we work with arbitrary permutations, not only cycles. The above proposition still holds true, with virtually the same proof. If we track down the proof of Theorem 5.11 from \cite{Pa3}, we see that the key ingredient was Theorem 1.2 of \cite{Fr}. Theorem 1.1 of the same article yields the following.

\begin{p}\label{P5.11}
There exists a constant $\mu$ such that for at least $(1-\mu /n)\cdot n!$
permutations $c$ the following holds: for any projection $p\in D_n$ with $0 < Tr(p)\leqslant 1/2$ we have $\lambda Tr(p)< d_H(p, apa^*)+d_H(p, cpc^*)$.
\end{p}

The same game we have to play with Theorem 5.20.

\begin{p}\cite[Theorem 5.20]{Pa3}
For any $\ve>0$ and nontrivial $\gamma\in\mathbb F_2$, there exists $n_0$ such that for any $n>n_0$ for at least $(1-\ve)[(n-1)!]$ $n$-cycles $c\in P_n$ we have 
$d_H\left(\gamma(a,c), 1_n\right)>1-\ve$.
\end{p}

Again, replacing Theorem 1.2 of \cite{Fr} with Theorem 1.1 of the same article, we get the following.

\begin{p}\label{T5.20}
For any $\ve>0$ and nontrivial $\gamma\in\mathbb F_2$, there exists $n_0$ such that for any $n>n_0$ for at least $(1-\ve)\cdot n!$ elements $c\in P_n$ we have 
$d_H\left(\gamma(a,c), 1_n\right)>1-\ve$.
\end{p}

In our argument we need two more results from \cite{Pa3}.

\begin{p}\cite[Proposition 5.12]{Pa3}\label{P5.12}
Let $\ve>0$ and $x\in P_n$. Then the number of permutations $y\in P_n$ such
that $d_H(x, y) <\ve$ is less than $n^{\lfloor n\ve\rfloor}$.
\end{p}

\begin{p}\cite[Proposition 5.13]{Pa3}\label{P5.13}
Let $\ve>0$. The number of permutations $y\in P_n$ such that $d_{H}(ay,ya)<\ve$ is less than $n^{\lfloor n\ve\rfloor+1}$.
\end{p}

When constructing many sofic representations of $\ff_2$, we have to make sure that are not conjugated. This is where the next definition comes into play. We shall take care of amplifications later. 

\begin{de}
For two collections of permutations in $P_n$, $(x_1,x_2,\ldots,x_k)$ and $(y_1,y_2,\ldots,y_k)$ define the distance:
\[d_S((x_1,x_2,\ldots,x_k),(y_1,y_2,\ldots,y_k))=\min_{p\in P_n}\sum_{i=1}^kd_H(x_i,py_ip^*).\]
\end{de}

We now prove a first result estimating the number of permutations $c\in P_n$ with certain properties. 

\begin{p}\label{S-distance}
Let $b\in P_n$. The number of permutations $c\in P_n$ such that $d_S((a,b),(a,c))<\lambda$ is less than $n^{2\lfloor n\lambda\rfloor+1}$.
\end{p}
\begin{proof}
Let $c\in P_n$ be such that $d_S((a,b),(a,c))<\lambda$. Then there exists $p\in P_n$ such that $d_H(a,pap^*)<\lambda$ and $d_H(b,pcp^*)<\lambda$. Let $A=\{q^*bq:q\in P_n, d_H(qa,aq)<\lambda\}$. It follows that $d_H(c,A)<\lambda$.

By Proposition \ref{P5.13}, the cardinality of $A$ is less than $n^{\lfloor n\lambda\rfloor+1}$. As $d_H(c,A)<\lambda$, by Proposition \ref{P5.12}, $c$ is in a set of cardinality at most $n^{\lfloor n\lambda\rfloor+1}\cdot n^{\lfloor n\lambda\rfloor}=n^{2\lfloor n\lambda\rfloor+1}$.
\end{proof}

Now we are ready to prove the main technical result, that will make the construction of uncountably many extreme points possible.

\begin{p}\label{k-permutations}
Let $k\in\nz^*$. There exists $n_k\in\nz$ and $k$ permutations $c_1,\ldots,c_k\in P_{n_k}$ such that 
\begin{enumerate}
\item $(a,c_i)$ is a $\lambda$-expander for each $i=1,\ldots,k$;
\item $d_S((a,c_i),(a,c_j))>\lambda$ for all $i\neq j$;
\item $Tr(\gamma(a,c_i))<1/k$ for all $i=1,\ldots,k$ and $\gamma\in B_k(\ff_2)$.
\end{enumerate}
\end{p}
\begin{proof}
Recall that $B_k(\ff_2)$ is the ball of radius $k$ around the origin, in the Cayley graph of $\ff_2$. We cannot ask condition $(3)$ for all elements of $\ff_2$, so we do it for larger and larger finite subsets. Standard procedure in the theory of sofic groups.

Fix $\ve>0$. By Propostion  \ref{P5.11} and Propostion \ref{T5.20}, applied for each $\gamma\in B_k(\ff_2)$, we get that for large enough $n$, for at least $(1-\ve)\cdot n!$ permutations $c\in P_n$, conditions $(1)$ and $(3)$ are satisfied. 

We must now choose $c_1,\ldots,c_k$ in this set, such that $d_S((a,c_i),(a,c_j))>\lambda$.
Here is where Proposition \ref{S-distance} comes in. Choosing a permutation $c\in P_n$ will exclude at most $n^{2\lfloor n\lambda\rfloor+1}$ other permutations. For large enough $n$, we have $k\cdot n^{2\lfloor n\lambda\rfloor+1}<(1-\ve)\cdot n!$. As such we can choose $c_1,\ldots, c_k$ with the required properties.
\end{proof}

\begin{te}
There exists $\Theta_i:\ff_2\to\Pi_{k\to\omega}P_{n_k}$, $i\in[0,1]$ sofic representations, such that each one is an expander, and $\Theta_i$, $\Theta_j$ are not conjugated for any $i\neq j$.
\end{te}
\begin{proof}
Use the previous proposition to construct a sequence $(n_k)_k$, and sets $A_k=\{c_k^1,\ldots,c_k^k\}\subset P_{n_k}$, with the specified properties.

In order to construct our sofic representations, we need an uncountable family $\mathcal{F}$ of infinite subsets of $\nz$, such that for $F_1,F_2\in\mathcal{F}$, $F_1\neq F_2$ we have $F_1\cap F_2$ is finite. An example was constructed in Remark 3.3 of \cite{Ca-Pa}. For $t\in[1/10,1)$, construct $F_t=\{\lfloor 10^kt\rfloor:k\in\nz^*\}$. It is easy to see that $\mathcal{F}=\{F_t:t\in[1/10,1)\}$ has the required property.

For $F\in\mathcal{F}$ construct a function $f:\nz\to\nz$, $f_F(k)=max\{i\in F:i\leqslant k\}$. Let $x_1,x_2$ be the two generators of $\ff_2$. For $F\in\mathcal{F}$ define a sofic morphism $\Theta_F:\ff_2\to\Pi_{k\to\omega}P_{n_k}$ by $\Theta_F(x_1)=\Pi_{k\to\omega}a_{n_k}$ and $\Theta_F(x_2)=\Pi_{k\to\omega}c_k^{f_F(k)}$. The third condition in Propostion \ref{k-permutations} ensures that $\Theta_F$ are sofic representations of $\ff_2$. The first condition implies that $\Theta_F$ is an expander. The second condition, together with the properties of the family $\mathcal{F}$ guarantee that $\{\Theta_F:F\in\mathcal{F}\}$ are not conjugated.
\end{proof}

We didn't specified until now, but expander easily implies being an extreme point.

\begin{p}\cite[Proposition 3.13]{Ar-Pa}
An expander sofic representation is an extreme point.
\end{p}

We constructed uncountably many extreme sofic representations, in the same sequence of dimensions, that are not conjugated. All that we have to do is to show that these sofic representations don't have amplifications that are conjugated. Getting rid of an amplification is no easy feat, and the problem deserve its own section.

\section{Representations with conjugated amplifications}

\begin{op}\label{open problem}
Let $\Theta_1,\Theta_2:\ff_2\to\Pi_{k\to\omega}P_{n_k}$ be two sofic representations such that $[\Theta_1]=[\Theta_2]$, i.e. there exists a sequence $(r_k)_k\in\nz^*$ and $p\in\Pi_{k\to\omega}P_{n_kr_k}$ such that $p(\Theta_1\otimes Id_{r_k})p^*=\Theta_2\otimes Id_{r_k}$. Then there exists $q\in\Pi_{k\to\omega}P_{n_k}$ such that $q\Theta_1q^*=\Theta_2$.
\end{op}

We suspect that this last statement is true, but it eluded our attempts. Let's see some particular examples where this is true. First of all, a similar statements holds for unitaries.

\begin{te}\cite[Theorem 6.4]{Ak1}
Let $\Theta_1,\Theta_2:G\to\Pi_{k\to\omega}\unit(n_k)$ be two morphisms such that there exists a sequence $(r_k)_k$ and $u\in\Pi_{k\to\omega}\unit(n_kr_k)$ such that $u(\Theta_1\otimes Id_{r_k})u^{-1}=\Theta_2\otimes Id_{r_k}$. Then there exists $v\in\Pi_{k\to\omega}\unit(n_k)$ such that $v\Theta_1v^{-1}=\Theta_2$.
\end{te}

Atkinson uses embeddings in type $II_1$ factors, as opposed to matrix algebras as we do here. However the proof, and the difficulty of the problem, are the same.

A first case we can solve for permutations is the one of a bounded sequence $(r_k)_k$.

\begin{p}\label{bounded amplification}
Let $\Theta_1,\Theta_2:\ff_2\to\Pi_{k\to\omega}P_{n_k}$ be two sofic representations such that there exists an $r\in\nz^*$ and $p\in\Pi_{k\to\omega}P_{n_kr}$ such that $p(\Theta_1\otimes Id_r)p^*=\Theta_2\otimes Id_r$. Then there exists $q\in\Pi_{k\to\omega}P_{n_k}$ such that $q\Theta_1q^*=\Theta_2$.
\end{p}

The proof is just a consequence of results in \cite{Pa-Ra}. We don't get into all details here, as we don't need this result, and it will take some time and space to familiarise the reader with the concepts of that article. In short, an element $p\in\Pi_{k\to\omega}P_{n_kr}$ will generate a DSE of multiplicity $r$ on the Loeb space $\Pi_{k\to\omega}D_{n_k}$. By Theorem 3.8 of \cite{Pa-Ra}, we can construct $q_\ve\in\Pi_{k\to\omega}P_{n_k}$ such that $d(q_\ve\Theta_1q_\ve^*,\Theta_2)<\ve$. The proof can be finished with a diagonal argument. 

We prove the result in the case of expander sofic representations, in order to finish the results in the previous section. The proof is based on results in Section 5.1 \cite{Pa3}. We first introduce some tools. 

\begin{de}
For $x,y\in M_n(\cz)$ define \emph{the Hamming distance} on matrices as:
\[d_H(x,y)=\frac1n|\{i:\exists j\ x(i,j)\neq y(i,j)\}|.\]
\end{de}

The formula counts the number of rows that are different in $x$ and $y$. If $x,y\in P_n$
then this is the usual Hamming distance on a symmetric group. We shall use this extended metric mainly on "pieces of permutations".

\begin{de}
A matrix $q\in M_n(\cz)$ is called a \emph{piece of permutation} if $q$ has only $0$ and $1$ entries and at most one entry of $1$ on each row and each column. Alternatively, $q=pa$ where $p\in P_n,$ and $a\in D_n$ is a projection.
\end{de}

\begin{p}
Let $x,y\in M_n(\cz)$ and $p\in P_n$. We still have bi-invariance:
\[d_H(x,y)=d_H(px,py)=d_H(xp,yp).\] If $p$ is a piece of permutation, then:
\[d_H(px,py)\leqslant d_H(x,y)\mbox{ and }d_H(xp,yp)\leqslant d_H(x,y).\]
\end{p}

The main technical part of the proof is contained in the next lemma.

\begin{p}
	Let $\varepsilon>0$, $n,r\in\nz^*$ and $x_1,\ldots, x_k, y_1\ldots, y_k\in P_n$ be such that $\{y_1\ldots, y_k\}$ is a $\lambda$-expander. Assume that there exists $u\in P_{nr}$ such that 
	$d_H(u(x_t \otimes 1_r),(y_t \otimes 1_r)u)<\varepsilon$ for all $t=1,\ldots,k$.  Then, there exists $v\in P_n$ such that $d_H(vx_t, y_t v)< 20k^2\varepsilon/\lambda$ for all $t=1,\ldots,k$.
\end{p}

\begin{proof}
The key point here is that the term $20k^2\ve/\lambda$ does not contain $r$. Otherwise, the problem would be solved using Propostion \ref{bounded amplification}.

By bi-invariance, we also have
	$d_H((x_t\otimes 1_r)u^*,u^*(y_t\otimes 1_r))<\varepsilon$ for all $t=1,\ldots,k$. 
	As $M_{nr}(\mathbb{C}) \simeq  M_r(\mathbb{C})\otimes M_n(\mathbb{C})$, elements in $M_{nr}(\mathbb{C})$ can be viewed as functions from $\{1,2,\ldots,r \}^2$
	to $M_n(\mathbb{C})$. Then $[u(x_t\otimes 1_r)](i,j)=u(i,j)x_t$ and $[(y_t\otimes 1_r)y](i,j)=y_t u(i,j)$. Also $u(i,j)$ is a piece of permutation for any $i,j$. If $A, B \in P_{nr}$ then: 
	$$2d_H(A,B)\geq \frac{1}{r}\sum_{i,j=1}^r d_H(A(i,j), B(i,j)).$$

	Let $\varepsilon^t_{i,j}=d_H(u(i,j)x_t,y_t u(i,j))$ and 
	 $\delta^t_{i,j}=d_H(x_tu^*(i,j), u^*(i,j)y_t)$. Then, for all  $t=1,\ldots,k$ we have
	\begin{align*}
	\frac{1}{r}\sum_{i,j=1}^r\varepsilon^t_{i,j} &\leq 2d_H(u(x_t\otimes 1_r), (y_t\otimes 1_r)u)<2\varepsilon, \\
	\frac{1}{r}\sum_{i,j=1}^r\delta^t_{i,j} &\leq 2d_H((x_t\otimes 1_r)u^*, u^*(y_t\otimes 1_r))<2\varepsilon. 
	\end{align*}
	From these inequalities we can deduce the existence of an $i\in \{1,2,\ldots, r\}$ such that  
	$$\sum_{j=1}^r \varepsilon^t_{i,j}< 4k\varepsilon,  \ \  \sum_{j=1}^r \delta^t_{j,i}< 4k\varepsilon$$
	for all  $t=1,\ldots,k.$ From now on $i$ is fixed with this property. For $t=1,\ldots, k$ we have
	\begin{align*}
	d_H(u(i,j)u^*(j,i)y_t, & y_tu(i,j)u^*(j,i)) \\
	&\leq d_H(u(i,j)u^*(j,i)y_t, u(i,j)x_tu^*(j,i)) +d_H(u(i,j)x_tu^*(j,i), y_tu(i,j)u^*(j,i))\\
	&\leq d_H(u^*(j,i)y_t, x_tu^*(j,i))+d_H(u(i,j)x_t, y_tu(i,j))=\delta^t_{j,i}+\varepsilon^t_{i,j}
	\end{align*}
	For $j=1,\ldots, r$, let $p_j=u(i,j)u^*(j,i)$. As $u(i,j)$ is a piece of permutation, $p_j$ is a projection in $D_n$. Moreover $\sum_{j=1}^rp_j=Id_n$, so these are disjoint projections. Also: $$d_H(p_j, y_tp_j y^*_t)=d_H(p_j y_t, y_tp_j )\leq \delta^t_{j,i}+\varepsilon^t_{i,j}.$$
	
	For $S\subset \{1,2,\ldots,r\}$, define $p_{S}=\sum_{j\in S}p_j$. Notice that
	$d_H(p_{S}, y_tp_{S}y_t^*)\leq \sum_{j\in S}d_H(p_j, y_tp_j y_t^*)$ and by the above inequalities, we get, for any subset $S$ and for any $t$
	$$d_H(p_{S}, y_t p_{S}y_t^*)\leq \sum_{j\in S} (\delta^t_{j,i}+\varepsilon^t_{i,j})<8k\varepsilon.$$

	By the expander hypothesis, we have
	$\lambda Tr(p_S)\leq\sum_{t=1}^k d_H(p_{S}, y_tp_{S}y_t^*)$. As such, if $Tr(p_S)\leq 1/2$, then actually $Tr(p_S)\leq 8k^2\varepsilon/\lambda$. As $\sum_{i=1}^r p_j= Id_n$, it is easy to see that there exists $j$ such that $Tr(p_j)\geq 1/2$. Let $S= \{1,2,\ldots,r\}\setminus{j}$. Then $Tr(p_S)< 8k^2\varepsilon/\lambda$, so $Tr(p_j)\geq 1-8k^2\varepsilon/\lambda$.
This means that $u(i,j)$ is ``almost" a permutation matrix. 	

	Let $v\in P_n$ be such that $d_H(v, u(i,j))<8k^2\varepsilon/\lambda$. For all $t=1,\ldots, r$ we get:  
	\begin{align*}
	d_H(vx_t,y_tv) & \leq d_H(vx_t,u(i,j)x_t)+d_H(u(i,j)x_t,y_tu(i,j))+d_H(y_tu(i,j),y_tv)\\
	&<8k^2\varepsilon/\lambda + 4k\varepsilon +8k^2\varepsilon/\lambda  < 20k^2\varepsilon/\lambda.
	\end{align*}
	
\end{proof}

\begin{te}
Let $\Theta_1,\Theta_2:\ff_2\to\Pi_{k\to\omega}P_{n_k}$ be two sofic representations such that $[\Theta_1]=[\Theta_2]$, i.e. there exists a sequence $(r_k)_k\in\nz^*$ and $p\in\Pi_{k\to\omega}P_{n_kr_k}$ such that $p(\Theta_1\otimes Id_{r_k})p^*=\Theta_2\otimes Id_{r_k}$. Assume that $\Theta_1$ is an expander. Then there exists $q\in\Pi_{k\to\omega}P_{n_k}$ such that $q\Theta_1q^*=\Theta_2$.
\end{te}
\begin{proof}
Let $\Theta_1(\gamma_i)=\Pi_{k\to\omega}x_i^k$, $i=1,2$ be such that $(x_1^k,x_2^k)$ is a $\lambda$-expander. Let $\Theta_2(\gamma_i)=\Pi_{k\to\omega}y_i^k$. Also choose $p_k$ such that $p=\Pi_{k\to\omega}p_k$. Apply the previous Proposition to $y_1^k,y_2^k,x_1^k,x_2^k$ and $p_k$, in order to get $q_k\in P_{n_k}$ with the stated property. Then $\Pi_{k\to\omega}q_k$ conjugates $\Theta_1$ to $\Theta_2$.
\end{proof}

We managed to prove this result because, when one conjugates an amplified expander, it has to be conjugated in bulks. An expander representation cannot be cut into pieces, and as such, inside the conjugating $p\in P_{nr}$, we found our $q\in P_n$ as one of the $r^2$ squares.

\section{Another strange sofic representation}\label{another}

%In the proof of Krein-Milman's theorem, in order to show the existence of extreme points, one uses compactness in two places: (1) to show that minimal faces consists of one point, and (2) in order to show that a decreasing chain of closed faces has non-empty intersection. In the previous section we showed that point (1) is solved for $Sof(G,P^\omega)$, even in the absence of compactness. Now we provide a strong indication that point (2) fails in general. 

In this section, we manufacture a sofic representation of $\ff_2$ that has trivial commutant in $\Pi_{k\to\omega}P_{n_k}$, but has uncountably many cuts. This suggests that this sofic representation has uncountably many, non-isomorphic cuts.

\subsection{Trivial commutant} Sofic representations with trivial commutant in permutations have been firstly constructed by Arzhantseva and second author in \cite{Ar-Pa}, Theorem 3.2. For the scope of the present paper we have to rework the proof, in order to get a better convergence for the counting argument used.  In \cite{Ar-Pa} it is shown that two permutations in $Sym(n)$ randomly chosen, have trivial commutant w.r.t. Hamming distance, with probability converging to $1$ as $n\to\infty$. The takeaway from this subsection is that the rate of this convergence is exponential. Recall that $a\in P_n$ is the matrix corresponding to the maximal cycle $i\to i+1$. We already used the estimation of number of permutations commuting with such a cycle.

\begin{p}\cite[Proposition 5.13]{Pa3}\label{P5.13}
Let $\ve>0$. The number of permutations $y\in P_n$ such that $d_{H}(ay,ya)<\ve$ is less than $n^{\lfloor n\ve\rfloor+1}$.
\end{p}

This time we also need an estimate for the number of permutations commuting with an arbitrary element.

\begin{p}\cite[Prop. 3.5]{Ar-Pa}\label{P:nr commuting}
Let $b\in P_n$ be such that $d_H(b, 1_n)>11\delta$. The number of permutations $c\in P_n$ such that $d_H(bc,cb)<\delta$ is less than $\frac{n!}{n^{4n\delta}}$, for large enough $n$.
\end{p}
\begin{proof}
Define $C=\{c\in P_n: d_H(bc,cb)<\delta\}$. Choose $c\in C$. Consider the following subsets of $\{1,\ldots,n\}$: $A_c=\{i:bc(i)=cb(i)\}$ and $B=\{i:b(i)\neq i\}$. Then $|A_c|>(1-\delta)n$ and $|B|>11\delta n$. It follows that $|A_c\cap B|>(11\delta-\delta)\cdot n=10\delta\cdot n$.

Let $i\in A_c\cap B$. Then $c\big(b(i)\big)=bc(i)$ and $b(i)\neq i$. Hence, once the value of $c(i)$ is fixed, the value of $c$ on $b(i)$ must be $bc(i)$. Unfortunately, the set $A_c$ depends on $c$. This makes the counting argument a little more involved.

Let's recall how to count the number of permutations $p\in P_n$: $p(1)$ can take any of the $n$ values in the set $\{1,\ldots,n\}$; $p(2)$ can take any of the remaining $n-1$ values, and so on. Hence, the cardinality of $P_n$ is $n!$. We adapt this argument to count the number of permutations $c$ with the required properties. Without loss of generality, we can assume that $B=\{1,2,\ldots,|B|\}$. As before $c(1)$ can take $n$ values. If $1\in A_c$, a information that at the moment we don't have, than $c(b(1))$ is also set. Thus, the following value of $c$ to be decided ($c(2)$ if $b(1)\neq 2$, and $c(3)$ otherwise), has only $n-2$ options. If $1\notin A_c$, we continue our enumeration of elements in $B$ till $|B|$. In the worst scenario, the first $\delta\cdot n$ elements of $B$ will not be in $A_c$. After this, all remaining elements of $B$ are bound to also be in $A_c$.

Thus, denoting by $t=\lfloor\delta n\rfloor$ and $s=\lfloor (11\delta-\delta)n/2\rfloor=\lfloor 5\delta\cdot n\rfloor$, our estimation for the maximal number of elements in $C$ is: $$\underbrace{n(n-1)\ldots(n-t+1)}_\text{$t$ terms}\underbrace{(n-t)(n-t-2)\ldots(n-t-2s+2)}_\text{$s$ terms}(n-t-2s)(n-t-2s-1)\ldots1.$$ Hence:
\[|C|<\frac{n!}{(n-t-2s+1)^s}<\frac{n!}{[(1-11\delta)n]^{5\delta n-1}}.\]
We only need to show that $[(1-11\delta)n]^{5\delta n-1}>n^{4n\delta}$. Using the logarithm, this is equivalent to:
\[(5\delta n-1)\ln[(1-11\delta)n]>4n\delta\cdot\ln{n}.\]
We factor the two terms and compute the limit via L'Hospital's rule.
\begin{align*}
&\lim_{n\to\infty}\frac{(5\delta n-1)\ln[(1-11\delta)n]}{4n\delta\cdot\ln{n}}=\lim_n\frac{5\delta\cdot\ln[(1-11\delta)n]+(5\delta n-1)\cdot1/n}{4\delta\cdot\ln{n}+4n\delta\cdot 1/n}=\\
&\lim_n\frac{5\delta\cdot\ln[(1-11\delta)n]}{4\delta\cdot\ln{n}}=\frac54>1.
\end{align*}
\end{proof}

We continue our counting argument by introducing two sets of $n$-cycles with specific properties. Given  $\delta >0$, we define:
\[L_n^\delta=\{c\in P_n:\not\exists b\in P_n\mbox{ with } d_H(b,1_n)>11\delta, d_H(ab,ba)<\delta\mbox{ and }d_H(cb,bc)<\delta\}.\]

\begin{p}\cite[Prop. 3.6]{Ar-Pa}\label{P:Gndelta}
For a fixed $\delta>0$ and large enough $n\in\nz$, $$Card\ L_n^\delta>(1-n^{-2\delta n})\cdot(n!).$$ 
\end{p}
\begin{proof}
According to Proposition \ref{P5.13} there are at most $n^{n\delta+1}$ permutations $b\in P_n$ such that $d_H(ab,ba)<\delta$. By Proposition \ref{P:nr commuting}, for each of those permutations $b$ with $d_H(b,1_n)>11\delta$, there are at most $n!\cdot n^{-4n\delta}$ cycles $c$ such that $d_H(cb,bc)<\delta$. All in all, the complement of $L_n^\delta$ has a cardinality less than $n^{n\delta+1}\cdot n!\cdot n^{-
4n\delta}<n^{-2 \delta n}\cdot n!$.  The conclusion hence follows.
\end{proof}

The set $L_n^\delta$ cannot be used directly to construct the required sofic representation. This is because $d_H(b,1_n)$ is in some sense a moving target, while in the definition of $L_n^\delta$ it is supposed to be fixed. This is why we introduce the following set:
\[K_n^\delta=\{c\in P_n:\forall b\in P_n,\ d_H(b,1_n)\leqslant 22\cdot max\{d_H(ab,ba),d_H(bc,cb),\delta\}.\}\]

\begin{p}\cite[Prop. 3.7]{Ar-Pa}\label{P:Hndelta}
For a fixed $1>\delta>0$ and a large enough $n\in\nz$, $$Card\ K_n^\delta>(1-\frac1{n^{\delta n}})[n!]$$ 
\end{p}
\begin{proof}
The proof is almost over when we notice that $K_n^\delta\supseteq L_n^\delta\cap L_n^{2\delta}\cap\ldots\cap L_n^{2^k\delta}$, where $k$ is minimal with the property that $2^{k+2}\delta>1$. So let $c\in L_n^\delta\cap L_n^{2\delta}\cap\ldots\cap L_n^{2^k\delta}$ and take $b\in P_n$. Denote by 
 $\lambda=max\{d_H(ab,ba),d_H(bc,cb)\}$. If $\lambda<\delta$, then as $c\in L_n^\delta$, $d_H(b,1_n)\leqslant 11\delta$.

Assume that $\lambda\geqslant\delta$. Then, there exists $i>0$ such that $2^{i-1}\delta\leqslant\lambda<2^i\delta$. If $i\leqslant k$, then $c\in L_n^{2^i\delta}$, so $d_H(b,1_n)\leqslant 11\cdot2^i\delta\leqslant 22\lambda$. If $i>k$ then $8\lambda>1$. This proves $c\in K_n^\delta$.

By Proposition \ref{P:Gndelta} and using De Morgan's formula $\cap_{j=0}^{k}L_n^{2^j\delta}=\overline{\cup_{j=0}^{k}\overline{L_n^{2^j\delta}}},$ we obtain that
$$|L_n^\delta\cap L_n^{2\delta}\cap\ldots\cap L_n^{2^k\delta}|>(1-\frac{k+1}{n^{2\delta n}})[n!].$$ As $k$ is fixed, depending only on $\delta$, for large enough $n$, $|K_n^\delta|>(1-\frac1{n^{\delta n}})[n!]$.
\end{proof}

\subsection{Permutations with small Coxeter lenght}
In the previous subsection we saw that most permutations are in the set $K_n^\delta$, that is needed to ensure a trivial commutant in permutations. Here we show that sufficiently many permutations can be cut into many pieces. The main tool is the \emph{Coxeter length}.

We first describe what we mean for a sofic representation to have many cuts. If an element $p\in\Pi_{k\to\omega}P_{n_k}$ commutes with all $\Pi_{k\to\omega}D_{n_k}$, then $p$ has to be identity. This is because $\Pi_{k\to\omega}D_{n_k}$ is a MASA in $\Pi_{k\to\omega}M_{n_k}$. Thus, we only ask for commutativity on a separable subalgebra. Let $(X,\mu)$ be the unit interval endowed with the Lebesgue measure. There is a canonic measure-presearving map, called \emph{the standard part}, from the Loeb space to $(X,\mu)$. This induces an embedding $St^*(L^\infty(X,\mu))\subset\Pi_{k\to\omega}D_{n_k}$, so $St^*(L^\infty(X,\mu))$ is a canonical separable von Neumann sublagebra in $\Pi_{k\to\omega}D_{n_k}$. For more information the reader should check Section 1 of \cite{CMP}. We don't need these details, we only need to understand when an element of $\Pi_{k\to\omega}P_{n_k}$ commutes with $St^*(L^\infty(X,\mu))$. Here is were the Coxeter length comes in.

\begin{de}
For $p\in P_n$, the \emph{Coxeter lenght} is defined as:
\[\ell_C(p)=\frac2{n(n-1)}Card\{i<j:p(i)>p(j)\}.\]
\end{de}

\begin{ex}\label{example}
For the $n$-cycle, $\ell_C(a_n)=\frac2{n(n-1)}\cdot (n-1)=\frac2n$. If $p(i)=n+1-i$, then $\ell_C(p)=1$, and this is the maximum value.
\end{ex}

The Coxeter function can be defined on elements of the universal sofic group, as an ultralimit. The only problem is that in this ultralimit it becomes just a semi-length. For details consult Section 4 of \cite{CMP}. For our discussion the Coxeter semi-length is relevant because of the following.

\begin{p}
An element $p\in\Pi_{k\to\omega}P_{n_k}$ commutes with $St^*(L^\infty(X,\mu))$ if and only if $\ell_C(p)=0$.
\end{p}

This is mostly Proposition 4.4 of \cite{CMP}, but one has to also check the definitions from Section 2 of the same article. Going back to our problem, we show that there are ``enough" permutations with small Coxeter length. We also use this opportunity to ensure a sofic representation of the free group. Thus, we are interested in estimations of the following set:

\[T_n^\delta=\{p\in P_n:\ell_C(p)<2\delta\mbox{ and } d_H(w(a_n,p),Id_n)>1-\delta \mbox{ for every $w\neq 1_{\ff_2}$ of length at most $1/\delta$}\}.\]

\begin{p}
For a fixed $1>\delta>0$ and a large enough $n\in\nz$, $Card\ T_n^\delta>\delta^n\cdot n!$. 
\end{p}
\begin{proof}
Let $k=\lfloor 1/\delta\rfloor$ and assume that $n=k\cdot m$. It might not be the case that $k|n$, but $k$ is fixed so the error will not affect our computations for large enough $n$. We consider only elements that permute the first m points, the next m points and so on. Any such permutation has total displacement less than $2\delta$:
\[\ell_C(p)\leqslant\frac2{n(n-1)}\cdot\frac{m(m-1)}2\cdot k=\frac{m-1}{mk-1}<\frac1k\leqslant2\delta.\]
We construct permutations in $T_n^\delta$ by choosing $k$ permutations $q\in P_m$, such that for each one of them $d_H(w(a_m,q),Id_m)>1-\delta \mbox{ for every $w\neq 1_{\ff_2}$ of length at most }1/\delta$. According to Proposition \ref{T5.20}, and as the ball of radius $1/\delta$ in $\ff_2$ is finite, for any $\ve>0$, for large enough $m$, there are at least $(1-\ve)\cdot m!$ permutations in $P_m$ satisfying this property.

We proved that $Card\ T_n^\delta>[(1-\ve)\cdot m!]^k$. We only need to prove that this value is larger than $\delta^n\cdot n!$, for large enough $n$. Recall Striling's approximation: $\sqrt{2\pi}\cdot\big(\frac ne\big)^n\cdot n\leqslant n!\leqslant e\cdot \big(\frac ne\big)^n\cdot n$. Setting $c=\sqrt{2\pi}$, we have:
\begin{align*}
[(1-\ve)(m!)]^k&\geqslant c^k(1-\ve)^k\cdot\big(\frac me\big)^{mk}\cdot m^k=c^k(1-\ve)^k\big(\frac 1k\big)^n\big(\frac ne\big)^n\cdot m^k\\
&> c^k(1-\ve)^k\delta^n\big(\frac ne\big)^n\cdot n\geqslant c^k(1-\ve)^k/e\cdot\delta^n\cdot n!
\end{align*}
We can choose $\ve>0$ small enough to guarantee that $c^k(1-\ve)^k/e>1$.

\end{proof}

Putting everything together, we get.

\begin{p}\label{p:special permutation}
For any $\delta>0$, for large enough $n$, $K_n^\delta\cap T_n^\delta\neq\emptyset$.
\end{p}
\begin{proof}
We need to show that in the limit $\delta^n\cdot n!>n^{-\delta n}\cdot n!$. This is equivalent to $\delta\cdot n^\delta>1$.
\end{proof}

\subsection{A sofic representation away from extreme points}

\begin{te}\label{strange rep}
There exists a sofic representation $\Theta:\ff_2\to\Pi_{k\to\omega}P_{n_k}$ such that $\Theta(\ff_2)^\prime\cap\Pi_{k\to\omega}P_{n_k}=\{Id\}$ and $St^*(L^\infty(X,\mu))\subset\Theta(\ff_2)^\prime$.
\end{te}
\begin{proof}
Let $(\delta_k)_k\in\rz_+^*$ be a decreasing sequence converging to $0$. By Proposition \ref{p:special permutation}, applied to each $\delta_k$, we construct a sequence of permutations $p_k\in P_{n_k}$ such that $p_k\in K_{n_k}^\delta\cap T_{n_k}^\delta$.
Set $\Theta(x_1)=\Pi_{k\to\omega}a_{n_k}$ and $\Theta(x_1)=\Pi_{k\to\omega}p_k$. Then, by \ref{example}, $\ell_C(\Theta(x_1))=\lim_{k\to\omega}2/n_k=0$ and $\ell_C(\Theta(x_2))=\lim_{k\to\omega}\ell_C(p_k)=0$. This shows that $St^*(L^\infty(X,\mu))\subset\Theta(\ff_2)^\prime$.

Let $b=\Pi_{k\to\omega}b_k\in\Pi_{k\to\omega}P_{n_k}$ be in the commutant of $\Theta$. 
This means that $\lim_{k\to\omega}d_H(a_{n_k}b_k,b_ka_{n_k})=0$ and $\lim_{k\to\omega}d_H(p_kb_k,b_kp_k)=0$. Let $\ve>0$ and take $F\in\omega$ such that for all $k\in F$, $\delta_k<\ve$ and $d_H(a_{n_k}b_k,b_ka_{n_k})<\ve$, $d_H(p_kb_k,b_kp_k)<\ve$. As $p_k\in K_{n_k}^{\delta_k}$, we get:
\[d_H(b_k,Id_{n_k})\leqslant22\cdot max\{d_H(a_{n_k}b_k,b_ka_{n_k}),d_H(p_kb_k,b_kp_k),\delta_k\}<22\ve.\]
As such, $d_H(b,Id)\leqslant22\ve$. As $\ve>0$ is arbitrary, $b=Id$.
\end{proof}

The reason that we believe $\Theta$ is not in the closure of the convex hull of extreme points in $Sof(\ff_2,P^\omega)$ is that for each $A$ a measurable subset of the unit interval, we can define $\Theta_A$ a cut of $\Theta$. The fact that $\Theta$ has no commutant in $\Pi_{k\to\omega}P_{n_k}$ suggest that these cuts are different as elements in $Sof(\ff_2,P^\omega)$. In order to prove this statement, one would need to solve conjecture \ref{open problem}. We end with another open problem, that would also finish this proof.

\begin{op}
Let $\Theta:G\to\Pi_{k\to\omega}P_{n_k}$ be a sofic representation in the closure of the convex hull of extreme points in $Sof(G,P^\omega)$. Then there are at most countably many different sofic representations that are obtained as cuts of $\Theta$.
\end{op}

\end{document}